\pdfoutput=1

\documentclass[12pt]{amsart}

\usepackage[bookmarks=true,%
    colorlinks=true,%
    linkcolor=blue,%
    citecolor=blue,%
    filecolor=blue,%
    menucolor=blue,%
    urlcolor=blue,%
    breaklinks=true]{hyperref}

\usepackage{bm}
\usepackage{amsthm, cleveref, amssymb, enumerate, multirow}
\usepackage{graphicx}
\usepackage{tikz-cd}

\newcommand{\Or}{\mathrm{O}}
\newcommand{\SO}{\mathrm{SO}}
\newcommand{\Vol}{\mathrm{Vol}}
\newcommand{\ab}{\mathrm{ab}}

\renewcommand{\H}{\mathbb{H}}

\newcommand{\Z}{\mathbb{Z}}
\newcommand{\R}{\mathbb{R}}
\newcommand{\PSL}{\mathrm{PSL}}
\newcommand{\Tor}{\mathrm{Tor}}

\newcommand{\PGL}{\mathrm{PGL}}
\newcommand{\GL}{\mathrm{GL}}

\newtheorem{theorem}{Theorem}[section]
\newtheorem{lemma}[theorem]{Lemma}
\newtheorem{corollary}[theorem]{Corollary}
\newtheorem{prop}[theorem]{Proposition}

\newtheorem{rem}[theorem]{Remark}
\newtheorem{quest}[theorem]{Question}

\crefname{theorem}{Theorem}{Theorems}
\crefname{lemma}{Lemma}{Lemmas}
\crefname{corollary}{Corollary}{Corollaries}
\crefname{prop}{Proposition}{Propositions}
\crefname{conj}{Conjecture}{Conjectures}
\crefname{rem}{Remark}{Remarks}
\crefname{quest}{Question}{Questions}
\crefname{section}{Section}{Sections}
\crefname{equation}{Equation}{Equations}
\crefname{figure}{Figure}{Figures}

\usepackage{color}

\title[3-manifolds in hyperbolic link complements of tori in $S^4$]{Totally geodesic hyperbolic 3-manifolds in hyperbolic link complements of tori in $S^4$}
\author{M. Chu}
\author{A. W. Reid}

\address{\newline
School of Mathematics,\newline
University of Minnesota,\newline
Minneapolis, MN 55455, USA.}
\email{mchu@umn.edu}
\address{\newline
Department of Mathematics,\newline
Rice University,\newline
Houston, TX 77005, USA}
\email{alan.reid@rice.edu}

\thanks{The first author is supported by NSF grant DMS $2243188$.}

\begin{document}

\begin{abstract}
In this paper we prove that certain hyperbolic link complements of $2$-tori in $S^4$ do not contain closed embedded totally geodesic hyperbolic $3$-manifolds.
\end{abstract}

%\subjclass{20E26, 20E18 (20F65, 20F10, 57M25) }

\keywords{cusped hyperbolic manifold, arithmetic manifold, totally geodesic submanifold}

\maketitle

%%%%%%%%%%%%%%%%%%%%%%%%%%%%%%%%%%%%%%%%%%%%%%%%%%%%%
%
%
% introduction
%
%
%%%%%%%%%%%%%%%%%%%%%%%%%%%%%%%%%%%%%%%%%%%%%%%%%%%%%

\section{Introduction}
\label{intro}

A classical problem in $4$-dimensional topology is (see \cite[Problem 3.20]{Kir}): {\em Under what conditions does a closed, orientable 3-manifold $M$ smoothly embed in $S^4$?} As an example of an obstruction, it is an old result 
of Hantzsche \cite{Han} that if a closed orientable 3-manifold $M$ embeds in $S^4$, then 
$\Tor(H_1(M,\Z))\cong A\oplus A$ for some finite abelian group $A$. The focus of this paper is obstructing the embedding of closed hyperbolic $3$-manifolds in $S^4$ via embeddings in hyperbolic link complements of $2$-tori in $S^4$.

A simple but elegant argument (see \cite[Proposition 4.10]{I0}) shows that if $X$ is a hyperbolic link complement of $2$-tori in $S^4$ then $\chi(X)=\chi(S^4)=2$, and so
there are only finitely many hyperbolic link complements of $2$-tori in $S^4$. This finiteness statement holds more generally for hyperbolic link complements of $2$-tori and Klein bottles in any fixed $4$-manifold.
By way of comparison, Thurston's hyperbolization theorem shows that many links in $S^3$ have hyperbolic complements, and although it is known that many hyperbolic link complements in $S^3$ do not contain a closed embedded totally geodesic surface (e.g. alternating links \cite{MeR}), examples do exist (see \cite{Lei} and \cite{MeR}).  The main result of this paper (see Theorem \ref{main} below) provides more
examples of hyperbolic link complements of $2$-tori in $S^4$
that do not contain a closed embedded totally geodesic hyperbolic $3$-manifold (our previous paper \cite{CR} provided one such example).  We note that \cite{CR} shows that the hyperbolic link complements of $2$-tori in $S^4$ in Theorem \ref{main}
do contain infinitely many immersed closed totally geodesic hyperbolic 3-manifolds. To state Theorem \ref{main} we need to recall some additional notation.

In \cite{RT} the authors provide a census of $1171$ so-called integral congruence two hyperbolic $4$-manifolds that are all obtained from face-pairings of the ideal $24$-cell in $\mathbb{H}^4$. These are all commensurable cusped, arithmetic, hyperbolic 4-manifolds of Euler characteristic $1$ (i.e. minimal volume). In \cite{I} the author provides an example of a cusped, orientable, hyperbolic 4-manifold of Euler characteristic $2$ that is the complement of five $2$-tori in $S^4$ (with the standard smooth structure \cite{I2}) and is constructed as the orientable double cover of the non-orientable manifold $1011$ in the census of integral congruence two hyperbolic $4$-manifolds mentioned above. In \cite{CR} we proved that this link complement of $2$-tori in $S^4$  
does not contain any closed embedded totally geodesic hyperbolic $3$-manifolds (it does contain embedded orientable non-compact finite volume totally geodesic hyperbolic $3$-manifolds).  In \cite{IRT}, four additional examples of link complements of $2$-tori in manifolds homeomorphic to $S^4$ were found. These arise as the orientable double covers of the non-orientable manifolds in the census of \cite{RT} with numbers $23$, $71$, $1091$ and $1092$.  The main result of this note is to extend the result of \cite{CR} to these four other examples. 
 
We fix the following notation. For $n\in \{23,71,1091,1092\}$, we denote by $p_n:W_n\rightarrow N_n$ the orientation double coverings of the non-orientable integral congruence two hyperbolic $4$-manifolds $N_n$. 
By construction, $\chi(N_n)=1$ and $\chi(W_n)=2$, with $W_n$ a link complement of $2$-tori in $S^4$.

\begin{theorem}
\label{main}
For $n\in \{23,71,1091,1092\}$ the manifolds $W_n$  do not contain a closed embedded totally geodesic hyperbolic $3$-manifold.
\end{theorem}

As with the case of the manifold $1011$ of \cite{RT}, the $W_n$ of Theorem \ref{main} all contain embedded non-compact finite volume totally geodesic hyperbolic $3$-manifolds.

The strategy of the proof of Theorem \ref{main} is similar to that of \cite{CR} but additional complications arise with these four examples (see \S \ref{tubular} for a fuller discussion). Moreover, a different argument is needed to handle these cases, and this requires a detailed analysis of possible closed totally geodesic surfaces that can embed in certain arithmetic hyperbolic $3$-manifolds that cover the Picard orbifold $\H^3/\PSL(2,\Z[i])$.

We finish the Introduction by posing a question prompted by our work:

\begin{quest}\label{question} Does there exist a hyperbolic link complement of $2$-tori in a closed (smooth) simply connected $4$-manifold that contains a closed embedded totally geodesic hyperbolic $3$-manifold?\end{quest}

Examples of hyperbolic link complements of $2$-tori in closed (smooth) simply connected $4$-manifolds are given in \cite{IRT} and \cite{Sar}. Indeed, \cite[Theorem 1.2]{Sar} shows that such link complements exist only in $S^4$, $\#_{r}(S^2\times S^2)$, or $\#_r (\mathbb{CP}^2 \# \overline{\mathbb{CP}^2})$, with $r>0$.  Furthermore, using the examples of \cite{I}, examples of link complements of $2$-tori in $\#_{r}(S^2\times S^2)$ for $r$ even were exhibited in \cite{Sar} (these cover the link complement of \cite{I}). It is unknown whether there exists a finite volume hyperbolic link complement of $2$-tori in $\#_r (\mathbb{CP}^2 \# \overline{\mathbb{CP}^2})$, for some $r>0$.

\section*{Acknowledgements}
{\em This work began whilst the authors were both visiting the Institut de Math\'ematiques, Universit\'e de Neuch\^atel, and we would like to thank the Insititut for its hospitality. We also wish to thank  A. Kolpakov, S. Riolo and L. Slavich for many helpful conversations on topics related to this work. We are particularly grateful to a referee who spotted some errors in the calculations in \S \ref{Fuchsian} and \ref{ss1092} of the first version of this paper, and made several helpful clarifying comments and suggestions.
}

\section{Recap from \texorpdfstring{\cite{CR}}{}}\label{recap}

The hyperboloid model of ${\mathbb H}^4$ is defined using the quadratic form $J = x_1^2+x_2^2+x_3^2+x_4^2-x_5^2$ with
\[{\mathbb H}^4 = \{x \in {\mathbb R}^{5} :J(x)= -1, x_{5}>0\}\]
equipped with the Riemannian metric induced from the Lorentzian inner product associated to $J$. The full group of isometries of ${\mathbb H}^4$ is then identified with $\Or^+(4,1)$, the subgroup of
\[
\Or(4,1) = \{A \in \GL(5,\mathbb{R}) : A^tJA = J\},
\]
preserving the upper sheet of the hyperboloid $J(x)=-1$, and where we abuse notation and use $J$ to denote
the symmetric matrix associated to the quadratic form. The full group of orientation-preserving isometries is given by $\SO^+(4,1) = \{A\in \Or^+(4,1) : \det(A)= 1\}$.

The groups $\Or^+(3,1)$ and $\SO^+(3,1)$ are defined in a similar manner.

\subsection{Integral congruence two hyperbolic 4-manifolds}\label{RTbackground}

The manifolds $p_n:W_n\rightarrow N_n$ where $n\in \{23,71,1091,1092\}$ of interest to us all arise as face-pairings of the regular ideal $24$-cell in ${\mathbb H}^4$ (with all dihedral angles $\pi/2$), and are regular $({\Z}/2{\Z})^4$ covers of the orbifold ${\H}^4/\Lambda(2)$ where $\Lambda(2)$ is the level two congruence subgroup of the group $\Or^+(J,{\Z}) = \Or^+(4,1) \cap \Or(J,{\Z})$. These manifolds are referred to
as {\em integral congruence two hyperbolic $4$-manifolds} in \cite[Table 1]{RT}. It will be useful to describe the $({\Z}/2{\Z})^4$ action, and this is best described in the ball model as follows.

Locate the $24$-cell in the ball model of hyperbolic space with vertices 
\[(\pm1,0,0,0), (0,\pm 1,0,0), (0,0,0\pm 1,0), (0,0,0,\pm 1) \] \[\text{ and }~(\pm\frac{1}{2}, \pm\frac{1}{2}, \pm\frac{1}{2}, \pm\frac{1}{2}).\]

The four reflections in the co-ordinate planes of ${\R}^4$ can be taken as generators of this $({\Z}/2{\Z})^4$ group of isometries. Passing to the hyperboloid model, these reflections are elements of $\Lambda(2)$ and are listed as the first four matrices in \cite[page 110]{RT}. Following \cite{RT} we denote
this $({\Z}/2{\Z})^4$ group of isometries by $K < \Lambda(2)$.

As noted in \cite{RT} (see also \cite{RT0}) all of the face-pairings of {\em any} of the integral congruence two hyperbolic $4$-manifolds are invariant under the group $K$. This implies that each of the coordinate hyperplane cross sections of the $24$-cell extends in each of the  integral congruence two hyperbolic $4$-manifolds to a totally geodesic hypersurface which is the fixed point set of one of the reflections described above. Following \cite{RT0} we call these hypersurfaces, {\em cross sections}.
As described in \cite{RT0}, these cross sections, can be identified with integral congruence two hyperbolic $3$-manifolds which are also described in \cite{RT}. Moreover, it is possible to use \cite{RT} or \cite{RT0} to identify these in any given example.

\begin{lemma}
\label{Xsection}
\begin{enumerate}
\item $N_{23}$ has $4$ non-orientable cross-sections all isometric to each other.
\item $N_{71}$  has one orientable cross-section isometric to the complement of the link $8^4_2$ and three non-orientable cross-sections.
\item $N_{1091}$ has one orientable cross-section isometric to the complement of the link $8^4_2$ and three non-orientable cross-sections, two of which are isometric to each other.
\item $N_{1092}$ has two orientable cross-sections isometric to the complement of the link $8^4_2$ and two non-orientable cross-sections which are isometric to each other.
\end{enumerate}
\end{lemma}

\begin{proof} The proof of this lemma is similar to the proof of \cite[Lemma 7.1]{CR}. We give a fairly detailed discussion of case (1), and only mention salient points in the remaining cases.

\begin{enumerate}
\item In \cite[Table 3]{RT} the manifold $N_{23}$ is given by the code \texttt{1569A4}  which represents the side pairing \texttt{1111555566669999AAAA4444} for the 24 sides of the ideal 24-cell $Q^4$. In the notation of \cite{RT}, the four cross sections have $k_1k_5k_9$ codes \texttt{352}, \texttt{352}, \texttt{156}, \texttt{156}, which correspond to the side pairings $r_ik_i$ for the 12 sides of the polytope $Q^3$ where $r_i$ is the reflection in side $i$ and $k_1=k_2=k_3=k_4$, $k_5=k_6=k_7=k_8$, $k_9=k_{10}=k_{11}=k_{12}$. Since $r_i$ is a reflection, the side pairing $r_ik_i$ is orientation preserving if and only if the corresponding $k_i$ is orientation reversing. But this happens only if $k_i\in\{1,2,4,7\}$ since then it corresponds to the diagonal matrices with $1\leftrightarrow\mathrm{diag}(-1,1,1,1)$, $2\leftrightarrow\mathrm{diag}(1,-1,1,1)$, $4\leftrightarrow\mathrm{diag}(1,1,-1,1)$, $7\leftrightarrow\mathrm{diag}(-1,-1,-1,1)$. Therefore, all four cross-sections of $N_{23}$ are non-orientable.

From \cite[Table 1]{RT}, we see that the code \texttt{156} corresponds to the non-orientable integral congruence two $3$-manifold $M_4^3$ of \cite{RT}. As in the proof of \cite[Lemma 7.1]{CR}, it can be checked that code \texttt{352} is equivalent to the code \texttt{156} via a symmetry of $Q^3$ (the polyhedron in \cite[Figure 2]{CR}), and hence determine isometric manifolds.

\item In \cite[Table 3]{RT} the manifold $N_{71}$ is given by the code \texttt{13EB34}. In the notation of \cite{RT}, the four cross sections have $k_1k_5k_9$ codes \texttt{712}, \texttt{152}, \texttt{173}, \texttt{136}. It can be checked that the $3$-manifold with code \texttt{712} is orientable and isometric to $M_{10}^3$ of \cite{RT}, and that the remaining codes determine non-orientable manifolds. Therefore, $N_{71}$ has one orientable cross-section, and three non-orientable cross-sections.

\item In \cite[Table 3]{RT} the manifold $N_{1091}$ is given by the code \texttt{53FF35}. In the notation of \cite{RT}, the four cross sections have $k_1k_5k_9$ codes \texttt{712}, \texttt{173}, \texttt{173}, \texttt{537}. As noted in case (2),
the code \texttt{712} determines an orientable $3$-manifold, and it can be checked that the remaining determine non-orientable ones. Hence, $N_{1091}$ has one orientable cross-section, and three non-orientable cross-sections.

\item In \cite[Table 3]{RT} the manifold $N_{1092}$ is given by the code \texttt{53FFCA}. In the notation of \cite{RT}, the four cross sections have $k_1k_5k_9$ codes \texttt{765}, \texttt{174}, \texttt{174}, \texttt{537}. Therefore, the manifolds
with code \texttt{174} is orientable and isometric to $M_{10}^3$ of \cite{RT}, and the remaining two determine non-orientable manifolds. Furthermore, it can be checked that the codes \texttt{765} and \texttt{537} are equivalent via a symmetry of $Q^3$ and determine isometric manifolds. 
Thus, $N_{1092}$ has two isometric orientable cross-sections, and two isometric non-orientable cross-sections.
\end{enumerate}
\end{proof}

In what follows, let $A$ (with orientable double cover $A^+$) denote the non-orientable manifold given by the code \texttt{537}, and similarly let $B$, $C$, $D$, and $E$ (with orientable double covers $B^+$, $C^+$, $D^+$, and $E^+$) denote the non-orientable manifolds given by the codes \texttt{152}, \texttt{173}, \texttt{136}, and \texttt{156} respectively.

\subsection{Volume from tubular neighbourhoods}
\label{tubular}
As in \cite{CR}, to prove \cref{main},  we will make use of a result of Basmajian \cite{Bas} which provides disjoint collars about closed embedded orientable totally geodesic hypersurfaces in hyperbolic manifolds.  We state this only for hyperbolic $4$-manifolds.

Following \cite{Bas}, let $r(x) = \log \coth(x/2)$, and let $V(r)$ denote the volume of a ball of radius $r$ in ${\H}^3$.  It is noted in \cite{Bas} that, $V(r) = \omega_3 \int_0^r \sinh^2(r)dr$, where $\omega_3$ is the area of the unit sphere in ${\R}^3$ (i.e. $\omega_3=4\pi$).

In \cite[pages 213--214]{Bas}, the volume of a tubular neighbourhood of a closed embedded orientable totally geodesic hyperbolic $3$-manifold of $3$-dimensional hyperbolic volume $A$ in a hyperbolic $4$-manifold is given
in terms of the the  $4$-dimensional tubular neighbourhood function $c_4(A) = (\frac{1}{2})(V \circ r)^{-1}(A)$. Moreover,
as noted in \cite[Remark 2.1]{Bas}, when the totally geodesic submanifold separates, an improved estimate can be obtained using the tubular neighbourhood function $d_4(A) = (\frac{1}{2})(V \circ r)^{-1}(A/2)$ and we record this as follows.

\begin{lemma}
\label{volume_in_hood} 
Let $X$ be an orientable finite volume hyperbolic $4$-manifold containing a closed embedded separating orientable totally geodesic hyperbolic $3$-manifold of $3$-dimensional hyperbolic volume $A$. 
Then $X$ contains a tubular neighbourhood of $M$ of volume
\[\mathcal{V}'(A) = 2A\int_0^{d_4(A)}\cosh^3(t) dt.\]
\end{lemma}

Moreover, \cite{Bas} also proves that disjoint embedded closed orientable totally geodesic hyperbolic $3$-manifolds in an orientable finite volume hyperbolic $4$-manifold have disjoint collars, thereby contributing additional volume.  For our purposes we summarize what we need in the following.

\begin{corollary}
\label{more_vol}
Let $X$ be an orientable finite volume hyperbolic $4$-manifold of Euler characteristic $\chi$ containing $K$ disjoint copies of a closed embedded orientable totally geodesic hyperbolic $3$-manifold of $3$-dimensional hyperbolic volume $A$. 
Assume that all of these disjoint copies separate $X$. Then
\[\Vol(X) = (\frac{4\pi^2}{3})\chi \geq K\mathcal{V}'(A).\]
\end{corollary}

Given this set up we recall the basic strategy of \cite{CR}.  
To that end, let $N$ (resp. $W$) denote one of the manifolds $N_{23}$, $N_{71}$, $N_{1091}$, or $N_{1092}$ (resp. $W_{23}$, $W_{71}$, $W_{1091}$, or $W_{1092}$) and $M\hookrightarrow W$ a closed embedded totally geodesic hyperbolic $3$-manifold. Since $W\subset S^4$, $M$ is orientable and the embedding separates $S^4$.

As in \cite[Lemma 7.2]{CR} since $W$ is the orientable double cover of $N$, it is a characteristic cover of $N$ and hence a regular cover of $\mathbb{H}^4/\Lambda(2)$ (using \cref{RTbackground}).
If it can be shown that $M$ is disjoint from the preimages 
of all of the cross-sections in $N$, then since
$W$ is a regular cover of ${\H}^4/\Lambda(2)$, using the isometries of $W$ induced from the reflections in the
co-ordinate hyperplanes we get $16$ disjoint copies of $M$, all embedded and separating in $W$ (since it is a submanifold of $S^4$). 

Now the minimal volume of a closed hyperbolic 3-manifold is that of the Weeks manifold and is approximately $0.9427\ldots$ \cite{GMM}.
Using this estimate for $\Vol(M)$,  and applying
\cref{more_vol} we see that $\Vol(W) \geq 16\mathcal{V}'(0.94)$, which  is approximately $28.9$. On the other hand, since $\chi(W)=2$, $\Vol(W)= \frac{8\pi^2}{3}$ which is approximately $26.3$, a contradiction.

As proved in \cite[Lemma 3.2]{CR}, any $M$ (as above) is disjoint from the lift of any orientable cross-section in $N$.
 To prove \cref{main} we need to show that $M$ is disjoint from the preimage of a non-orientable cross-section in $N$.
 This follows from our next lemma, since if $M$ (as above) was not disjoint this would give rise to a closed embedded totally geodesic (possibly non-orientable) surface in the 
preimage of the cross-section. 

\begin{lemma}
\label{noclosedinNon}
Let $Y$ be any of the non-orientable cross-sections listed in \cref{Xsection} and $Y^+$ the orientable double cover. Then $Y^+$ (and hence $Y$) does not contain a closed embedded totally geodesic surface.
\end{lemma}

The strategy to prove \cref{noclosedinNon} is this: we first identify $\Gamma=\pi_1(Y^+)$ as a congruence subgroup of the Picard group $\PSL(2,\Z[i])$, and identify matrices (up to sign) that correspond to a generating set for $\Gamma$. We next use the classification of circles left invariant by non-elementary Fuchsian subgroups of $\PSL(2,\Z[i])$ given in \cite{MR0} to limit the possibilities for what circles can be associated to a closed embedded totally geodesic surface in $Y^+$. Finally we use a criterion given by \cite[Corollary 3.3]{JR} to prove that any candidate totally geodesic surface cannot be embedded. 

The proof of \cref{noclosedinNon}  occupies the remainder of this paper.

\section{The Picard group and the fundamental groups of the cross-sections}\label{picard}

The Picard group $\PSL(2,\Z[i])$ has a presentation from \cite{Sw}:
\begin{multline*}
\PSL(2,\Z[i]) = <\alpha,l,t,u|\alpha^2=l^2=(\alpha l)^2=(tl)^2=(ul)^2\\=(\alpha t)^3=(u\alpha l)^3=[t,u]=1>,
\end{multline*}
where these generators can be represented by the matrices (up to sign) shown below.  
\[\alpha=\left(
\begin{array}{cc}
 0 & 1 \\
 -1 & 0 \\
\end{array}
\right),
t=\left(
\begin{array}{cc}
 1 & 1 \\
 0 & 1 \\
\end{array}
\right),
u=\left(
\begin{array}{cc}
 1 & i \\
 0 & 1 \\
\end{array}
\right),
l=\left(
\begin{array}{cc}
 i & 0 \\
 0 & -i \\
\end{array}
\right).\]

\subsection{Locating subgroups}
\label{locate}
It will be helpful to prove the following result which helps "locate" the fundamental groups of the manifolds $Y^+$ as subgroups of $\PSL(2,\Z[i])$.

\begin{prop}\label{inpicard}
Let $Y$ be any of the non-orientable cross-sections listed in \cref{Xsection} and $Y^+$ the orientable double cover.  Then $\pi_1(Y^+)$ admits a faithful representation with image group $\Gamma$ of index $48$ in $\PSL(2,\Z[i])$ and
$[\Gamma(1+i),\Gamma(1+i)] \triangleleft \Gamma \triangleleft  \Gamma(1+i)$.\end{prop}

\begin{proof}  To establish that $\pi_1(Y^+)$ admits a faithful representation with image group $\Gamma$ with $\Gamma \triangleleft  \Gamma(1+i)$, 
recall from \cite[Section 3]{RT} that these integral congruence two hyperbolic 3-manifolds are constructed as follows.  As in the case of dimension $4$ described above these manifolds arise as regular covers (all with covering group $(\Z/2\Z)^3$) of a
certain congruence quotient of $\H^3$, namely the subgroup $\Lambda(2) < \Or^+(3,1;\Z)$. As shown in \cite[p.105]{RT}:
\[[\Or^+(3,1;\Z):\Lambda(2)]=12~\hbox{with}~\Or^+(3,1;\Z)/\Lambda(2) \cong S_3\times \Z/2\Z.\] 
In addition, \cite[Section 3]{RT} identifies the group $\Or^+(3,1;\Z)$ with the Coxeter group $T$, generated by reflections in the faces of the non-compact tetrahedron with
Coxeter diagram:
\medskip
\begin{center}
\begin{tikzcd}
 \bullet \arrow[r,"3", start anchor={[xshift=-2ex]}, end anchor={[xshift=2ex]}, dash] & \bullet \arrow[r,"4", start anchor={[xshift=-2ex]}, end anchor={[xshift=2ex]}, dash] & \bullet \arrow[r,"4", start anchor={[xshift=-2ex]}, end anchor={[xshift=2ex]}, dash] & \bullet
\end{tikzcd}
\end{center}

\medskip

Using a presentation of this Coxeter group, one can find a presentation for the subgroup $\SO^+(3,1;\Z)$ (consisting of orientation preserving isometries) of index $2$ in $\Or^+(3,1;\Z)$, and it follows from this that the abelianization 
of $\SO^+(3,1;\Z)$ is 
$\Z/2\Z\times\Z/2\Z$. Indeed, it is known (see \cite{BLW} for example) that the subgroup $\SO^+(3,1;\Z)$ can be identified with $\PGL(2,\Z[i])$, which in turn contains $\PSL(2,\Z[i])$ of index $2$.

Since $\Lambda(2)$ contains elements of determinant $-1$, the group $\Lambda^+(2)=\Lambda(2)\cap\SO^+(3,1;\Z)$ has index $2$ in $\Lambda(2)$, and so $\Lambda^+(2)$ is isomorphic to a normal subgroup of index $12$ in 
$\PGL(2,\Z[i])$ with quotient group $S_3\times \Z/2\Z$. 

%% M: changed some of the wording here to appease the referee
As noted above, the abelianization of $\SO^+(3,1;\Z)\cong \PGL(2,\Z[i])$ is $\Z/2\Z\times\Z/2\Z$. 
We claim that this implies that $\Lambda^+(2)$ is contained in the commutator subgroup of $\SO^+(3,1;\Z)$. To see this, first recall that the abelianization of $S_3$ is $\Z/2\Z$ so we get the epimorphism $a:S_3\times \Z/2\Z\rightarrow\Z/2\Z\times\Z/2\Z$. Since the surjective composition
\[\SO^+(3,1;\Z)\xrightarrow{p} S_3\times \Z/2\Z\xrightarrow{a} \Z/2\Z\times\Z/2\Z,\]
lands in an abelian group, it factors through the abelianization of $\SO^+(3,1;\Z)$, with the second map an isomorphism. It then follows that $\Lambda^+(2)$ must be contained in the commutator subgroup of $\SO^+(3,1;\Z)$.

Furthermore, via the above identifications, $\Lambda^+(2)$ must be isomorphic to a normal subgroup of index $6$ in $\PSL(2,\Z[i])$. By \cite[Theorem 2]{FN} there is a unique normal subgroup of index $6$ in $\PSL(2,\Z[i])$ and it is the principal congruence subgroup $\Gamma(1+i)$ (i.e. those elements in $\PSL(2,\Z[i])$ congruent to the identity modulo the ideal $(1+i)$).

From \cite[p.105]{RT}, the group $\Or^+(3,1;\Z)$ is generated by the 4 reflections
\[
a= \begin{pmatrix}
0&1&0&0 \\ 1&0&0&0 \\ 0&0&1&0 \\ 0&0&0&1
\end{pmatrix}, 
b= \begin{pmatrix}
1&0&0&0 \\ 0&0&1&0 \\ 0&1&0&0 \\ 0&0&0&1
\end{pmatrix}, 
c= \begin{pmatrix}
1&0&0&0 \\ 0&1&0&0 \\ 0&0&-1&0 \\ 0&0&0&1
\end{pmatrix},\]
\[d= \begin{pmatrix}
0&-1&-1&1 \\ -1&0&-1&1 \\ -1&-1&0&1 \\ -1&-1&-1&2
\end{pmatrix}, 
\]
and the subgroup $\SO^+(3,1;\Z)$ can be identified with $\PGL(2,\Z[i])$ via the isomorphism defined by
\[
ab \mapsto \begin{pmatrix} 1&-1 \\ 1&0 \end{pmatrix},
ac \mapsto \begin{pmatrix} 0&1 \\ -i&0 \end{pmatrix},
ad \mapsto \begin{pmatrix} 0&1 \\ 1&0 \end{pmatrix}.
\]

From \cite[p.106]{RT} the group $\Lambda(2)$ is generated by the reflections 
$r_1=abcba$, $r_2=bcb$, $r_3=c$, $r_4=abdcdba$, $r_5=bdcdb$, $r_6=dcd$, and the subgroup $\Lambda^+(2)$ can therefore be identified with the group $\Gamma(1+i)$ via the induced isomorphism defined by
\[
r_1r_2 \mapsto \begin{pmatrix} -1&1+i \\ -1+i&1 \end{pmatrix},
r_1r_3 \mapsto \begin{pmatrix} -i&-1+i \\ 0&i \end{pmatrix},
r_1r_4 \mapsto \begin{pmatrix} -i&2i \\ 0&i \end{pmatrix},
\] \[
r_1r_5 \mapsto \begin{pmatrix} 2-i&-1+i \\ 1-i&i \end{pmatrix},
r_1r_6 \mapsto \begin{pmatrix} 1&1+i \\ 0&1 \end{pmatrix}.
\]

As noted above, $\pi_1(Y) \triangleleft \Lambda(2)$ with quotient group $(\Z/2\Z)^3$, so it follows that $\pi_1(Y^+) \triangleleft \Lambda^+(2)\cong\Gamma(1+i)$ with quotient group $(\Z/2\Z)^3$.
Now $\Gamma(1+i)/[\Gamma(1+i),\Gamma(1+i)]\cong (\Z/2\Z)^5$ (see the Magma \cite{Mag} routine following this proof). Since the map $\Gamma(1+i) \rightarrow \Gamma(1+i)/\pi_1(Y^+)\cong (\Z/2\Z)^3$ goes to an abelian group, the commutator subgroup $[\Gamma(1+i),\Gamma(1+i)]$ is sent to $1$ and it follows that $[\Gamma(1+i),\Gamma(1+i)]\triangleleft \pi_1(Y^+)$.  

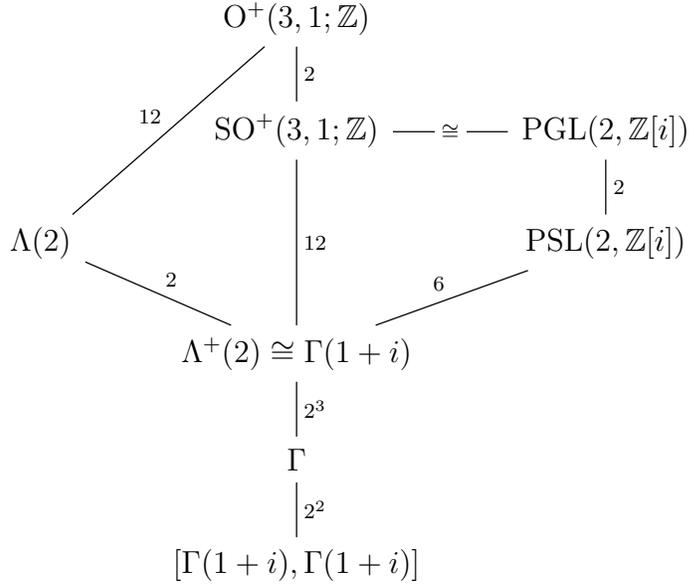
\begin{figure}
\begin{tikzcd}
& \Or^+(3,1;\Z) \arrow[d,"2",dash] \arrow[ldd,"12"',dash] & & \\
& \SO^+(3,1;\Z) \arrow[r,"\cong " description,dash] \arrow[dd,"12",dash] & \PGL(2,\Z[i]) \arrow[d,"2",dash] \\
\Lambda(2) \arrow[rd,"2",dash] & & \PSL(2,\Z[i]) \arrow[dl,"6"',dash] \\
& \Lambda^+(2) \cong \Gamma(1+i) \arrow[d,"2^3",dash] & \\
& \Gamma \arrow[d,"2^2",dash] & \\
& \left[ \Gamma(1+i),\Gamma(1+i) \right]  &
\end{tikzcd}
\caption{Lattice of Subgroups} \label{lattice}
\end{figure}
Putting all of this together we obtain the lattice of subgroups shown in \cref{lattice}.
We now have the image group $\Gamma=\pi_1(Y^+)$ as claimed. That the index in $\PSL(2,\Z[i])$ is $48$ is clear from the lattice of subgroups above (or from volume consideration discussed in \cite[p.108]{RT}).
 \end{proof}

\begin{rem}\label{iscommutator}
The subgroup $[\Gamma(1+i),\Gamma(1+i)]$ can also be identified as the principal congruence subgroup $\Gamma(2+2i)$ (i.e. those elements in $\PSL(2,\Z[i])$ congruent to the identity modulo the ideal $(2+2i)=(1+i)^3$).
This follows from \cite[Propositon 3.1]{BR} where the group $\Gamma(2+2i)$ is identified as a link group, and arises as the normal closure in $\PSL(2,\Z[i])$ of the subgroup $<t^2u^2,t^4>$. Now $t^2u^2,t^4\in \Gamma(1+i)$, and
since $\Gamma(1+i)^{\ab}\cong (\Z/2\Z)^5$, it follows that $t^2u^2$ and $t^4$ are both mapped trivially under the abelianization map $\Gamma(1+i)\rightarrow \Gamma(1+i)^{\ab}$. Hence
$<t^2u^2, t^4> \subset [\Gamma(1+i),\Gamma(1+i)]$. 

Now the subgroup $[\Gamma(1+i),\Gamma(1+i)]$ is a characteristic subgroup of $\Gamma(1+i)$ and hence is normal in $\PSL(2,\Z[i])$, and so it
follows that the normal closure in $\PSL(2,\Z[i])$ of the subgroup $<t^2u^2,t^4>$ is contained in $[\Gamma(1+i),\Gamma(1+i)]$; i.e. $\Gamma(2+2i) \subset [\Gamma(1+i),\Gamma(1+i)]$. However, both these groups have index
$192$ in $\PSL(2,\Z[i])$.\end{rem}

We now provide the short Magma \cite{Mag} routine referred to in the proof of \cref{inpicard}. Referring below, the group {\tt{g}} is the group $\PSL(2,\Z[i])$, and the presentation used is that given above. 
The group {\tt{N}} is $\Gamma(1+i)$. 

\medskip

\begin{verbatim}
> g<a,l,t,u>:=Group<a,l,t,u|a^2,l^2,(a*l)^2,(t*l)^2,
(u*l)^2,(a*t)^3,(u*a*l)^3,(t,u)>;
> h:=sub<g|t*u,t^2,u^2>;
> N:=NormalClosure(g,h);
> print Index(g,N);
6
> print AbelianQuotientInvariants(N); 
[ 2, 2, 2, 2, 2 ]
\end{verbatim}

\subsection{Generators for the fundamental groups of the cross-sections}\label{pres}
Using the description of each non-orientable cross-section from
\cref{Xsection}, the notation of \S \ref{RTbackground} together with the isomorphism described in the proof of \cref{inpicard}, we find generators for the fundamental groups of $A$, $B$, $C$, $D$ and $E$ as side pairings of the 24-cell. We then use Magma to find generators for the orientation double covers $A^+$, $B^+$, $C^+$, $D^+$ and $E^+$ and also to eliminate redundant generators. The generators are written below both as elements of $\PSL(2,\Z[i])$ and as words in the reflections $r_i$'s.

\begin{multline*} \pi_1(A^+)= \left\langle 
\begin{pmatrix} -i&-1+i \\ 1+i&-i \end{pmatrix},
\begin{pmatrix} 3i&1-i \\ 1+i&-i \end{pmatrix},
\begin{pmatrix} 1-2i&-2-4i \\ -2i&-3-2i \end{pmatrix}, \right.
\\ \left.
\begin{pmatrix} -1-4i&-4+4i \\ -2-2i&-1+4i \end{pmatrix},
\begin{pmatrix} -2-5i&-1+9i \\ -3-3i&2+7i \end{pmatrix}
 \right\rangle 
 \\ = \left\langle r_3r_6r_2r_1, r_1r_3r_6r_2, r_2r_1r_5r_3r_4r_3, r_2r_1r_5r_3r_2r_5r_3r_1, \right.
 \\ \left. r_2r_1r_5r_1r_6r_2r_3r_5r_1r_2 \right\rangle ,
\end{multline*}
\begin{multline*} \pi_1(B^+)= \left\langle 
\begin{pmatrix} 3&-1-i \\ 2-2i&-1 \end{pmatrix},
\begin{pmatrix} -2+i&1+i \\ 1+i&-i \end{pmatrix},
\begin{pmatrix} 1+2i&2-2i \\ 4i&5-2i \end{pmatrix}, \right.
\\ \left.
\begin{pmatrix} 1+2i&-1-i \\ 4+4i&-3-2i \end{pmatrix},
\begin{pmatrix} -2-i&1+i \\ -1-i&i \end{pmatrix}
 \right\rangle 
  \\ = \left\langle r_1r_2r_6r_2, r_3r_4r_3r_2, (r_5r_3r_1)^2, r_5r_3r_2r_6r_2r_1r_3r_5, r_5r_1r_4r_2r_1r_5 \right\rangle ,
\end{multline*}
\begin{multline*} \pi_1(C^+)= \left\langle 
\begin{pmatrix} 1&-1-i \\ 2-2i&-3 \end{pmatrix},
\begin{pmatrix} -1&1+i \\ -2&1+2i \end{pmatrix},
\begin{pmatrix} 3+2i&-2-2i \\ 2+2i&-1-2i \end{pmatrix}, \right. 
\\ \left.
\begin{pmatrix} 1&1+i \\ 0&1 \end{pmatrix}, 
\begin{pmatrix} -3&5-3i \\ -2&3-2i \end{pmatrix}
 \right\rangle 
  \\ = \left\langle r_2r_6r_2r_1, r_2r_5r_3r_1, (r_1r_4r_2)^2, r_1r_4r_6r_1r_4r_1, r_1r_4r_5r_3r_1r_2r_4r_1 \right\rangle ,
\end{multline*}
\begin{multline*} \pi_1(D^+)= \left\langle 
\begin{pmatrix} 1&-1-i \\ 2-2i&-3 \end{pmatrix},
\begin{pmatrix} -2+3i&-2i \\ -1+i&-i \end{pmatrix},
\begin{pmatrix} 2+i&-2 \\ 1-i&i \end{pmatrix},\right.
\\ \left.
\begin{pmatrix} 2+i&-2i \\ 3+i&-3i \end{pmatrix}, 
\begin{pmatrix} -3+2i&4 \\ -5+i&5+2i \end{pmatrix}
 \right\rangle 
  \\ = \left\langle r_2r_6r_2r_1, r_2r_3r_2r_3r_4r_3r_2r_1r_5r_2, r_2r_3r_2r_4r_2r_1r_5r_2, r_2r_5r_1r_2r_3r_2r_3r_4r_3r_2, \right.
  \\ \left. r_2r_5r_1r_2r_6r_2r_5r_1 \right\rangle ,
\end{multline*}
\begin{multline*} \pi_1(E^+)= \left\langle 
\begin{pmatrix} 1&-1-i \\ 2-2i&-3 \end{pmatrix},
\begin{pmatrix} 3&-2 \\ 2&-1 \end{pmatrix},
\begin{pmatrix} 1+2i&2-2i \\ 4i&5-2i \end{pmatrix},\right.
\\ \left.
\begin{pmatrix} -1+2i&2 \\ -4+2i&3+2i \end{pmatrix}, 
\begin{pmatrix} 3+2i&-1-i \\ 4+4i&-1-2i \end{pmatrix}
 \right\rangle 
  \\ = \left\langle  r_2r_6r_2r_1, r_4r_3r_2r_1r_3r_5, (r_5r_3r_1)^2, r_5r_3r_1r_4r_3r_2, r_5r_3r_1r_2r_6r_2r_3r_5 \right\rangle .
\end{multline*}

\section{Fuchsian subgroups of the Picard group}\label{Fuchsian}

To understand totally geodesic surfaces in $Y^+$,  we recall from \cite{MR0} that every non-elementary Fuchsian subgroup of $\PSL(2,\Z[i])$ fixes a circle or straight line $\mathcal{C}$
with an equation of the form:~$a|z|^2+\bar{B}z+B\bar{z}+c=0$, with $a,c \in \mathbb{Z}$ and $B \in \Z[i]$, and vice versa. We caution the reader that the normalization of the equation of the circle follows \cite{JR} (which we will use in the proof) rather than \cite{MR0}, the normalization of \cite{MR0} uses $Bz+\bar{B}\bar{z}$.

Two such circles (or straight-lines) $\mathcal{C}$ and $\mathcal{C}'$ are said to be {\em equivalent} if there exists 
$\gamma\in\PSL(2,\Z[i])$ such that $\gamma \mathcal{C}=\mathcal{C}'$.
Define $D=|B|^2-ac$ to be the {\em discriminant} of $\mathcal{C}$. This is preserved by the action of  $\PSL(2,\Z[i])$, and hence equivalent circles have the same discriminant (see \cite{MR0}).  If $\Delta <\PSL(2,\Z[i])$ is a torsion-free subgroup of finite index, then a $\Delta$-equivalence class of circles and straight-lines can be associated to a totally geodesic surface in $\mathbb{H}^3/\Delta$ and vice versa.  Hence
we can also refer to the discriminant of  the associated totally geodesic surface.
When $a\neq0$, $\mathcal{C}$ is a circle centered at $-B/a$, with radius $\sqrt{D}/|a|$.  This is the case when the totally geodesic surface is closed \cite[Lemma 3.1]{JR}.  Note that if the surface associated to a circle $\mathcal{C}$ is closed and embedded in $\mathbb{H}^3/\Delta$, then for every element $\delta\in\Delta$ we must have $\delta\mathcal{C}=\mathcal{C}$ or $\delta\mathcal{C}\cap\mathcal{C}=\emptyset$.

It is shown in \cite{MR0} that every circle (or straight-line) as above is $\PSL(2,\Z[i])$-equivalent to one of the following:

\begin{itemize}
\item $\mathcal{C}_D~:~|z|^2-D=0$,
\item $\mathcal{C}_{D,1}~:~2|z|^2+z+\bar{z}-\frac{D-1}{2}=0$ (when $D\equiv 1 \pmod{4}$),
\item $\mathcal{C}_{D,2}~:~2|z|^2+iz-i\bar{z}-\frac{D-1}{2}=0$ (when $D\equiv 1 \pmod{4}$), or
\item $\mathcal{C}_{D,3}~:~2|z|^2+(1+i)z+(1-i)\bar{z}-\frac{D-2}{2}=0$ (when $D\equiv 2 \pmod{4}$).
\end{itemize}
Note that the radius of the first circle listed above is $\sqrt{D}$ and for the others it is $\sqrt{D}/2$.

\begin{lemma}
\label{standard_embed}
Let $\Delta$ be a normal subgroup of finite index in $\PSL(2,\Z[i])$, and assume that $S\hookrightarrow \H^3/\Delta$ is an embedded totally geodesic surface associated to the circle $\mathcal{C}$ of discriminant $D$. Then there exists an embedded totally geodesic surface $S'$ associated to one of the circles $\mathcal{C}_D$ or $\mathcal{C}_{D,j}$ for one of $j=1,2,3$.\end{lemma}

\begin{proof} Observe that, since $\mathcal{C}$ is associated to an embedded totally geodesic surface in $\mathbb{H}^3/\Delta$,  for any $\alpha\in \PSL(2,\Z[i])$, the circle $\alpha \mathcal{C}$ is associated to an embedded totally geodesic surface in $\mathbb{H}^3/\alpha\Delta\alpha^{-1}$. Since $\Delta$ is assumed to be normal in $\PSL(2,\Z[i])$, the surface associated to $\alpha \mathcal{C}$ is actually embedded in $\H^3/\Delta$.

From the classification of circles given above, there is  $\alpha\in \PSL(2,\Z[i])$ such that $\alpha \mathcal{C}$ is one of $\mathcal{C}_D$ or $\mathcal{C}_{D,j}$ for one of $j=1,2,3$. The result follows.\end{proof}

Associated to the circle $\mathcal{C}$ with equation $a|z|^2+\bar{B}z+B\bar{z}+c=0$ as above, is the Hermitian matrix 
$A=\left(
\begin{array}{cc}
 a & B \\
 \overline{B} & c \\
\end{array}
\right)$ 
with an action of $\PSL(2,\Z[i])$ given by $\gamma^* A \gamma$, where $*$ denotes conjugate-transpose and the given action sends $\mathcal{C}$ to $\gamma^{-1} \mathcal{C}$. 
Here $\mathcal{C}$ is the set of all points $z\in\mathbb{C}$ such that $A\begin{pmatrix}z\\1 \end{pmatrix}\cdot\begin{pmatrix}\overline{z}\\1 \end{pmatrix}=0$.
Now \cite[Corollary 3.3]{JR} provides a criterion for a totally geodesic surface $S$ associated to a circle $\mathcal{C}$ to be embedded in $\mathbb{H}^3/\Delta$ or not; namely if $\gamma\in\Delta$ does not leave $\mathcal{C}$ invariant and satisfies:
\[
\left|\mathrm{tr}(\gamma^*A\gamma A^{-1})\right| < 2,
\]
then $\gamma^{-1} \mathcal{C} \cap  \mathcal{C}\neq\emptyset$ and $S$ is not embedded. Furthermore, if $S$ is closed, and $\gamma$ does not leave $\mathcal{C}$ invariant, then $\gamma^{-1} \mathcal{C} \cap  \mathcal{C}$ is two points.

We will make use of the following lemma.

\begin{lemma}
\label{onlycircle}
Let $\Delta$ be a subgroup of finite index in $\PSL(2,\Z[i])$ which contains the group $[\Gamma(1+i),\Gamma(1+i)]$, and let $M=\H^3/\Delta$. If $S\hookrightarrow M$ is a closed embedded totally geodesic surface (not necessarily orientable) associated to a circle
$\mathcal{C}$, then there exists $\alpha\in \PSL(2,\Z[i])$ such that $\alpha \mathcal{C} =  \mathcal{C}_{6,3}$.\end{lemma}

\begin{proof} 
Recall from \cref{iscommutator} that $[\Gamma(1+i),\Gamma(1+i)]=\Gamma(2+2i)$, and so $\Gamma(2+2i) \subset \Delta$ by hypothesis. Therefore $S$ gives rise to a closed embedded totally geodesic surface associated to the 
circle $\mathcal{C}$ in the cover $\mathbb{H}^3/\Gamma(2+2i)$.  Assuming that $\mathcal{C}$ has discriminant $D$, there exists $\alpha\in \PSL(2,\Z[i])$ such that
$\alpha \mathcal{C}$ is one of $\mathcal{C}_D$ or $\mathcal{C}_{D,j}$ for $j=1,2,3$ 

Now $\Gamma(2+2i)\triangleleft \PSL(2,\Z[i])$, and \cref{standard_embed} shows that one of 
$\mathcal{C}_D$ or $\mathcal{C}_{D,j}$ for $j=1,2,3$ also gives rise to a closed embedded totally geodesic surface in $\mathbb{H}^3/\Gamma(2+2i)$.

Now the element $\begin{pmatrix}
 1 & 2+2 i \\
 0 & 1
\end{pmatrix} \in\Gamma(2+2i)$, and using this element, a simple calculation shows that for the surface to be embedded, the radius of the associated circle must be $\leq \sqrt{2}$.  
From above, the radii of the circles $\mathcal{C}_D$ or $\mathcal{C}_{D,j}$ for $j=1,2,3$ is $\sqrt{D}$ or $\sqrt{D}/2$. Hence, amongst the circles $\mathcal{C}_D$ or $\mathcal{C}_{D,j}$, $j=1,2,3$, the only possibilities are:
\[\mathcal{C}_1,\mathcal{C}_2, \mathcal{C}_{1,1}, \mathcal{C}_{1,2}, \mathcal{C}_{2,3}, \mathcal{C}_{5,1}, \mathcal{C}_{5,2}, \mathcal{C}_{6,3}\]
and the only one of these that can give rise to a {\em closed} surface is $\mathcal{C}_{6,3}$ (see \cite[Lemma 8]{MR0}).  

The upshot of this discussion is that if $S\hookrightarrow M$ is a closed embedded totally geodesic surface with associated circle $\mathcal{C}$, then there exists $\alpha\in \PSL(2,\Z[i])$ such that $\mathcal{C}=\alpha\mathcal{C}_{6,3}$.\end{proof}

The proof of Lemma \ref{noclosedinNon} will be completed in the sections below.  To that end we make some additional comments and introduce some notation.  From Proposition \ref{inpicard} and \cref{iscommutator}, we need to consider certain groups $\Delta$ with  $\Gamma(2+2i) \triangleleft \Delta \triangleleft  \Gamma(1+i)$.  

A complete system of (left  or right) coset representatives for $\Gamma(1+i)$ in $\PSL(2,\Z[i])$ is provided by the following $6$ matrices:
\[T_0=\rm{id},
T_1=\begin{pmatrix}
 0 & 1 \\
 -1 & 0 \end{pmatrix},
T_2=\begin{pmatrix}
 1 & -1 \\
 0 & 1 \end{pmatrix},\]
\[T_3=\begin{pmatrix}
 1 & 0 \\
 -1 & 1 \end{pmatrix},
T_4=\begin{pmatrix}
 1 & 1 \\
 -1 & 0 \end{pmatrix},
T_5=\begin{pmatrix}
 0 & -1 \\
 1 & 1 \end{pmatrix}.\]
Using this and the normality of $\Delta$ in $\Gamma(1+i)$ (see proof of \cref{standard_embed}), it follows that if  $\mathcal{C}=\alpha\mathcal{C}_{6,3}$ corresponds to a closed embedded totally geodesic surface in $\H^3/\Delta$ then one of 
$T_i\mathcal{C}_{6,3}$ also corresponds to such a surface.  
Briefly, since $\mathcal{C}$ corresponds to a closed embedded totally geodesic surface in $\H^3/\Delta$, $\mathcal{C}_{6,3}=\alpha^{-1}\mathcal{C}$ corresponds to a closed embedded totally geodesic surface in $\H^3/\alpha^{-1}\Delta\alpha$.
Writing $\alpha=\gamma T_i$ for some $\gamma\in \Gamma(1+i)$, and using $\Delta \triangleleft \Gamma(1+i)$ we deduce that $\mathcal{C}_{6,3}$ corresponds to a closed embedded totally geodesic surface in $\H^3/T_i^{-1}\Delta T_i$
from which it follows that $T_i\mathcal{C}_{6,3}$ corresponds to a closed embedded totally geodesic surface in $\H^3/\Delta$.

Using the action on the Hermitian forms described above, the action by the matrices $T_i$ is given by $(T_i^{-1})^*AT_i^{-1}$, which, since the entries of the matrices $T_i$ are integers, is simply $(T_i^{-1})^tAT_i^{-1}$. Hence,
the circles $T_i\mathcal{C}_{6,3}$ for $i=0,1,\ldots ,5$ are represented by the matrices: 
\[A=A_0=\begin{pmatrix}
 2 & 1-i \\
 1+i & -2 \end{pmatrix},
A_1=\begin{pmatrix}
 -2 & -1-i \\
 -1+i & 2 \end{pmatrix},\]
\[A_2=\begin{pmatrix}
 2 & 3-i \\
 3+i & 2 \end{pmatrix},
A_3=\begin{pmatrix}
 2 & -1-i \\
 -1+i & -2 \end{pmatrix},\]
\[A_4=\begin{pmatrix}
 -2 & -3-i \\
 -3+i & -2 \end{pmatrix},
A_5=\begin{pmatrix}
 -2 & 1-i \\
 1+i & 2 \end{pmatrix}.\]

\section{Proving no closed embedded totally geodesic surfaces} 
\label{ss1092}

From Lemma \ref{Xsection} to prove \cref{noclosedinNon} (i.e. that the link complements $W_{23}$, $W_{71}$, $W_{1091}$ and $W_{1092}$ do not contain a closed embedded totally geodesic hyperbolic $3$-manifold), we are reduced to showing that the hyperbolic $3$-manifolds $A^+$, $B^+$, $C^+$, $D^+$ and $E^+$ do not contain a closed embedded totally geodesic surface (which could be non-orientable). 

In what follows in each of the subsections below we list elements of the groups $\pi_1(A^+)$, $\pi_1(B^+)$, $\pi_1(C^+)$, $\pi_1(D^+)$ and $\pi_1(E^+)$ that provide self-intersections of the circles 
$T_i\mathcal{C}_{6,3}$. This is done using the matrix generators for each of 
$\pi_1(A^+)$, $\pi_1(B^+)$, $\pi_1(C^+)$, $\pi_1(D^+)$ and $\pi_1(E^+)$ listed in \S \ref{pres} and the  criteria of \cite{JR} stated
\S \ref{Fuchsian}:
\[
\left|\mathrm{tr}(\gamma^*A_i\gamma A_i^{-1})\right| < 2.
\]
These calculations were performed in Mathematica \cite{Math} and the notebook is available from the authors upon request. 
For convenience, we shall simply denote the generators for each of  the groups $\pi_1(A^+)$, $\pi_1(B^+)$, $\pi_1(C^+)$, $\pi_1(D^+)$ and $\pi_1(E^+)$ in \S \ref{pres} by $g_1,g_2,\ldots ,g_5$ in the order that they are listed.  What is listed below are the Hermitian forms $A=A_0,A_1, \ldots , A_5$ and
those elements $\gamma$, written in terms of $g_1,g_2,\ldots ,g_5$, for which $\left|\mathrm{tr}(\gamma^*A_i\gamma A_i^{-1})\right| < 2$. 

We will also make use of the element 
$l = \begin{pmatrix}
 i & 0 \\
 0 & -i \end{pmatrix},$ which being an element of $\Gamma(1+i)$ normalizes each of the groups $\pi_1(A^+)$, $\pi_1(B^+)$, $\pi_1(C^+)$, $\pi_1(D^+)$ and $\pi_1(E^+)$ by \cref{inpicard}.
Additional explanation of elements that are not visibly in the groups $\pi_1(A^+)$, $\pi_1(B^+)$, $\pi_1(C^+)$, $\pi_1(D^+)$ and $\pi_1(E^+)$ is provided when needed.

Finally, we remark that we also need to ensure that the elements do not leave the circles in question invariant.  This is clear if the elements are parabolic (since the surface is closed) and when the trace is a non-real complex number that is not purely imaginary.  In the cases where the element has trace that is pure imaginary we 
check to see whether the circle is left invariant. \\[\baselineskip]
%
%
%\vfill\eject
\noindent{\bf The manifold $A^+$:}

\begin{table}[ht]
\begin{center}
\begin{tabular}{|c|c|c|}
\hline
Circle/Form&Element&$|\hbox{Trace Value}|$\\
\hline
$A$&$g_3$&$2/3$\\
\hline
$A_1$&$g_1$&$2/3$\\
\hline
$A_2$&$g_1g4$&$2/3$\\
\hline
$A_3$&$g_1$&$2/3$\\
\hline
$A_4$&$l\beta l^{-1}$&$2/3$\\
\hline
$A_5$&$g_3$&$2/3$\\
\hline
\end{tabular}
\end{center}
\end{table}

\noindent Note that the element $g_1$ has trace $-2i$. However, a calculation shows that $g_1$ does not leave invariant any of the circles $T_i\mathcal{C}_{6,3}$ for $i=0, \dots ,5$.  
That $\beta\in \pi_1(A^+)$, can be checked by noting that 
\begin{align*}
g_2^{-1}g_1g_3\beta &= \begin{pmatrix}
 5+8 i & 10-2 i \\
 18+6 i & 13-16 i \end{pmatrix} \\
&= \begin{pmatrix}  1 + 4(1+2 i )& (2+2i)(2-3i) \\
 (2+2i)(6-3i) & 1+ 4(3-4 i) \end{pmatrix}
 \end{align*} 
is in $\Gamma(2+2i)$. By \cref{iscommutator} and \cref{inpicard}, $\Gamma(2+2i)=[\Gamma(1+i),\Gamma(1+i)]<\pi_1(A^+)$.

None of the elements $g_3$, $g_1g_4$ and $\beta$ (and hence also $l\beta l^{-1}\in \pi_1(A^+)$) have  purely imaginary trace.\\[\baselineskip]
\noindent{\bf The manifold $B^+$:}

\begin{table}[ht]
\begin{center}
\begin{tabular}{|c|c|c|}
\hline
Circle/Form&Element&$|\hbox{Trace Value}|$\\
\hline
$A$&$g_1$&$2/3$\\
\hline
$A_1$&$g_1$&$2/3$\\
\hline
$A_2$&$g_2$&$2/3$\\
\hline
$A_3$&$g_2$&$2/3$\\
\hline
$A_4$&$l g_1l^{-1}$&$2/3$\\
\hline
$A_5$&$g_1$&$2/3$\\
\hline
\end{tabular}
\end{center}
\end{table}

\noindent From \S \ref{pres}, $g_1$ and $g_2$ are both parabolic, and as above $l g_1l^{-1}\in \pi_1(B^+)$.\\[\baselineskip]
\noindent{\bf The manifold $C^+$:}~In this case the parabolic element $g_4$ works for all the forms $A, A_1,\ldots ,A_5$ with trace value $2/3$.\\[\baselineskip]
%
%
%\vfill\eject
\noindent{\bf The manifold $D^+$:}~

\begin{table}[ht]
\begin{center}
\begin{tabular}{|c|c|c|}
\hline
Circle/Form&Element&$|\hbox{Trace Value}|$\\
\hline
$A$&$g_1$&$2/3$\\
\hline
$A_1$&$g_1$&$2/3$\\
\hline
$A_2$&$l g_1l^{-1}$&$2/3$\\
\hline
$A_3$&$l g_1l^{-1}$&$2/3$\\
\hline
$A_4$&$l g_1l^{-1}$&$2/3$\\
\hline
$A_5$&$l g_1l^{-1}$&$2/3$\\
\hline
\end{tabular}
\end{center}
\end{table}

\noindent From \S \ref{pres}, $g_1$ is parabolic, and as noted above $l g_1l^{-1}\in \pi_1(D^+)$.\\[\baselineskip]
\noindent{\bf  The manifold $E^+$:}~Since the parabolic element $g_1\in\pi_1(E^+)$ is exactly the same as for $\pi_1(D^+)$, the same table holds for $E^+$ as that shown for $D^+$.

\begin{rem}\label{final} From \cref{iscommutator} and \cref{inpicard} we know that $\H^3/\Gamma(2+2i)$ covers each of the manifolds 
$A^+$, $B^+$, $C^+$, $D^+$, and $E^+$, which we have shown do not contain a closed embedded totally geodesic surface. On the other hand, as pointed out in \cite{JR} the link complement $\H^3/\Gamma(2+2i)$ does contain a closed totally geodesic surface of genus $3$ associated to $\mathcal{C}_{6,3}$.\end{rem}

%\bibliography{refs}

\begin{thebibliography}{CGLS}

%\bibitem[ALR]{ALR} I. Agol, D. D. Long and A. W. Reid, {\em The Bianchi groups are separable on geometrically finite subgroups}, Annals of Math. {\bf 153} (2001), 599--621.

\bibitem{BR} M. D. Baker and A. W. Reid, {\em Principal congruence link complements},  Ann. Fac. Sci. Toulouse {\bf 23} (2014), 1063--1092.

\bibitem{Bas} A. Basmajian, {\em Tubular neighborhoods of totally geodesic hypersurfaces in hyperbolic manifolds}, Invent. Math. {\bf 117} (1994), 207--225.

%\bibitem[BHCh]{BHCh} A. Borel and Harish-Chandra, {\em Arithmetic subgroups of algebraic groups}, Ann. Math. {\bf 75} (1962), 485--535.

\bibitem{Mag}  W.~Bosma, J.~Cannon, and C.~Playoust, \emph{The Magma algebra system. I. The user language}, J. Symbolic Comput.,  {\bf 24} (1997), 235--265.

\bibitem{BLW}  A. M. Brunner, Y. W.  Lee, and N. J. Wielenberg, {\em Polyhedral groups and graph amalgamation products}, Topology Appl. {\bf 20} (1985), 289--304.

%\bibitem[BB]{BB} R. Budney and B. A. Burton,  {\em Embeddings of 3-manifolds in $S^4$ from the point of view of the $11$-tetrahedron census}, arXiv:0810.2346.

%\bibitem[CM]{CM} C. Cao and G. R. Meyerhoff, {\em The orientable cusped hyperbolic 3-manifolds of minimum volume}, Invent. Math. {\bf 146} (2001), 451--478. 

\bibitem{CR} M.Chu and A. W. Reid, {\em Embedding closed hyperbolic 3-manifolds in small volume hyperbolic 4-manifolds},  Algebraic \& Geometric Topology, {\bf 21} (2021), 2627--2647.

%\bibitem[CGLS]{CGLS} M. Culler, C. McA. Gordon, J. Luecke and P. B. Shalen, {\em Dehn surgery on knots}, Annals of Math. {\bf 125} (1987), 237--300.

%\bibitem[CJR]{CJR} M. Culler, W. Jaco, and H. Rubinstein, {\em Incompressible surfaces in once-punctured torus bundles}, Proc. London Math. Soc. {\bf 45} (1982), 385--419.

%\bibitem[Dix]{Dix} L. E. Dixon, {\em Linear Groups, with an Exposition of the Galois Field Theory},  Dover Publications, New York (1958).

\bibitem{FN} B. Fine and M. Newman, {\em The normal subgroup structure of the Picard group}, Trans. Amer. Math. Soc. {\bf 302} (1987), 769--786.

%\bibitem[FH]{FH} W. Floyd and A. Hatcher, {\em Incompressible surfaces in punctured-torus bundles}, Topology Appl. {\bf 13} (1982), 263--282.

\bibitem{GMM} D. Gabai, R. Meyerhoff and P. Milley, {\em Minimum volume cusped hyperbolic three-manifolds}, J. Amer. Math. Soc. {\bf 22} (2009), 1157--1215.

\bibitem{Han} W. Hantzsche, {\em Einlagerung von Mannigflitigkeiten in euklidishe Raume}, Math. Zeit. {\bf 43} (1938), 38--58.

%\bibitem[HT]{HT} A. Hatcher and W. Thurston, {\em Incompressible surfaces in $2$-bridge knot complements}, Invent. Math. {\bf 79} (1985), 225--246.

\bibitem{I0} D. Ivansic, {\em Embeddability of noncompact hyperbolic manifolds as complements of codimension-1 and -2 submanifolds}, In memory of T. Benny Rushing. Topology Appl. {\bf 120} (2002), 211--236. 

\bibitem{I} D. Ivansic, {\em Hyperbolic structure on a complement of tori in the $4$-sphere}, Adv. Geom. {\bf 4} (2004), 119--139.

\bibitem{I2}  D. Ivansic, {\em A topological 4-sphere that is standard}, Adv. Geom. {\bf 12} (2012), 461--482.

\bibitem{IRT} D. Ivansic, J. G. Ratcliffe and S. T. Tschantz, {\em Complements of tori and Klein bottles in the $4$-sphere that have hyperbolic structure}, Algebraic and Geometric Topology, {\bf 5} (2005), 999--1026.

\bibitem{JR} J. Jung and A. W. Reid, {\em Embedding closed totally geodesic surfaces in Bianchi orbifolds}, to appear Math. Research Letters,  {\tt{arXiv:2003.05427}}.

%\bibitem[Ki0]{Kir0} R. C. Kirby, {\em The Topology of $4$-Manifolds}, L.N.M. {\bf 1374}, Springer-Verlag (1989).

\bibitem{Kir} R. Kirby, {\em Problems in low-dimensional topology}, Edited by R. Kirby. AMS/IP Stud. Adv. Math., 2.2, Geometric topology (Athens, GA, $1993$), 35--473, Amer. Math. Soc., Providence, RI, (1997).

%\bibitem[KM]{KM} S. Kojima and Y. Miyamoto, {\em The smallest hyperbolic 3-manifolds with totally geodesic boundary}, J. Diff. Geom. {\bf 34} (1991), 175--192. 

%\bibitem[KRS]{KRS} A. Kolpakov, A. W. Reid, and L. Slavich, {\em Embedding arithmetic hyperbolic manifolds}, Math. Res. Lett. {\bf 25} (2018), 1305--1328.

\bibitem{Lei} C. J. Leininger, {\em Small curvature surfaces in hyperbolic 3-manifolds}, J. Knot Theory Ramifications {\bf 15} (2006), 379--411.

%\bibitem[Loz]{Loz} M-T. Lozano, {\em Arcbodies}, Math. Proc. Camb. Phil. Soc. {\bf 94} (1983), 253--260.

\bibitem{MR0} C. Maclachlan and A. W. Reid, {\em Parameterizing Fuchsian subgroups of Bianchi groups}, Canadian J. Math. {\bf 43} (1991), 158--181.

\bibitem{MR} C. Maclachlan and A. W. Reid, {\em The Arithmetic of Hyperbolic $3$-Manifolds}, Graduate Texts in Mathematics, {\bf 219}, Springer-Verlag (2003).

\bibitem{MeR} W. Menasco and A. W. Reid, {\em Totally geodesic surfaces in hyperbolic link complements}, in Topology $'90$ (Columbus, OH, $1990$), 215--226, Ohio State Univ. Math. Res. Inst. Publ. {\bf 1}, de Gruyter, Berlin, (1992).

%\bibitem[Mil]{Mil} J. J. Millson, {\em On the first Betti number of a constant negatively curved manifold}, Annals of Math. {\bf 104} (1976), 235--247.

%\bibitem[Mon]{Mon} J. M. Montesinos-Amilibia, {\em On odd rank integral quadratic forms: canonical representatives of projective classes and explicit construction of integral classes with square-free determinant}, Rev. R. Acad. Cienc. Exactas Fís. Nat. Ser. A Mat. RACSAM {\bf 109} (2015), 199--245.

%\bibitem[Oe]{Oe} U. Oertel, {\em Closed incompressible surfaces in complements of star links}, Pacific J. Math. {\bf 111} (1984), 209--230.

\bibitem{RT0} J. G. Ratcliffe and S. T. Tschantz, {\em Gravitational instantons of constant curvature}, in Topology of the Universe Conference (Cleveland, OH, 1997), Classical Quantum Gravity {\bf 15} (1998), 2613--2627.

\bibitem{RT} J. G. Ratcliffe and S. T. Tschantz, {\em The volume spectrum of hyperbolic $4$-manifolds}, Experimental Math. {\bf 9} (2000), 101--125.

%\bibitem[RiSl]{RiSl} S. Riolo and L. Slavich, {\em New hyperbolic $4$-manifolds of low volume}, Algebraic and Geometric Topology, {\bf 19} (2019), 2653--2676.

\bibitem{Sar} H. Saratchandran, {\em Finite volume hyperbolic complements of $2$-tori and Klein bottles in closed smooth simply connected $4$-manifolds}, New York J. Math. {\bf 24} (2018), 443--450. 

%\bibitem[Se]{Se}  J-P. Serre, {\em A Course in Arithmetic}, Graduate Texts in Math. {\bf 7} Springer-Verlag (1973).

\bibitem{Sw}  R.G. Swan, {\em Generators and relations for certain special linear groups}, Advances in Math. {\bf 6} (1971), 1--77.

%\bibitem[Th]{Th} W. P. Thurston, {\em The Geometry and Topology of 3-Manifolds}, Princeton University mimeographed notes, (1979).

%\bibitem[VS]{VS} E. B. Vinberg and O. V. Shvartsman, {\em Discrete groups of motions of spaces of constant curvature}, in Geometry II, Encyc. Math. Sci. {\bf 29} Springer (1993), 139--248.

%\bibitem[Wu]{Wu} Y-Q. Wu, {\em Incompressibility of surfaces in surgered $3$-manifolds}, Topology {\bf 31} (1992), 271--279.

\bibitem{Math} Wolfram Research, Inc., {\em Mathematica}, Version 11.2. Champaign, IL (2017).

\end{thebibliography}
%\bibliographystyle{alpha}

%\bibitem[BCP97]{Mag} W. Bosma, J. Cannon, and C. Playoust, {\em The Magma algebra system. I. The user language}, J. Symbolic Comput.,  {\bf 24} (1997), 235--265.

%\bibitem[CDGW]{CDGW} M. Culler, N. M. Dunfield, M. Goerner, and J. R. Weeks, {\em SnapPy, a computer program for studying the geometry and topology of 3-manifolds}, \url{http://snappy.computop.org}, version 2.6 (12/29/2017).

\end{document}